\documentclass[12pt]{amsart}
\usepackage{amsmath}
\usepackage{hyperref}
\usepackage{amssymb}
\usepackage{latexsym}
\usepackage{amscd}
\usepackage{xcolor}
\usepackage{graphicx} 
\usepackage{amsthm}
\usepackage{mathrsfs}
\usepackage{xypic}
\usepackage{bm}

\newdimen\AAdi%
\newbox\AAbo%
\def\AAk#1#2{\s_etbox\AAbo=\hbox{#2}\AAdi=\wd\AAbo\kern#1\AAdi{}}%
\def\AAr#1#2#3{\s_etbox\AAbo=\hbox{#2}\AAdi=\ht\AAbo\raise#1\AAdi\hbox{#3}}%
\font\tenmsb=msbm10 at 12pt \font\sevenmsb=msbm7 at 8pt
\font\fivemsb=msbm5 at 6pt
\newfam\msbfam
\textfont\msbfam=\tenmsb \scriptfont\msbfam=\sevenmsb
\scriptscriptfont\msbfam=\fivemsb
\def\Bbb#1{{\tenmsb\fam\msbfam#1}}
\newtheorem{thm}{Theorem}[section]

\newtheorem{cor}{Corollary}[section]

\newtheorem{pro}{Proposition}[section]

\newcommand{\ba}{\begin{array}}
\newcommand{\ea}{\end{array}}

\newcommand{\Section}[2]{\setcounter{equation}{0}
\allowdisplaybreaks
\section[#1]{#2}}

\def\n{\nabla}

\def\f#1#2{\frac{#1}{#2}}

\def\mc#1{\mathcal{#1}}

\def\pf#1{\frac{\partial}{\partial #1}}
\def\pd#1#2{\frac {\partial #1}{\partial #2}}

\def\a{\alpha}
\def\be{\beta}

\def\p#1{\partial #1}

\def\de{\delta}
\def\De{\Delta}
\def\e{\eta}
\def\ep{\varepsilon}

\def\g{\gamma}

\def\la{\lambda}

\def\om{\omega}
\def\Om{\Omega}
\def\th{\theta}

\def\w{\wedge}

\def\ze{\zeta}

\def\R{\Bbb{R}}
\def\C{\Bbb{C}}

\def\lan{\langle}
\def\ran{\rangle}
\def\ra{\rightarrow}

\subjclass{58E20,53A10.}

\begin{document}
\title
[Berstein type theorems] {Bernstein type theorems for
spacelike stationary graphs in Minkowski spaces}

\author
[Xiang Ma, Peng Wang and Ling Yang]{Xiang Ma, Peng Wang and Ling Yang}
\address{School of Mathematical Sciences, Peking University, Beijing 100871, China}
\email{maxiang@math.pku.edu.cn}
\address {Department of Mathmatics, Tongji University,
Shanghai 200092, China.} \email{netwangpeng@tongji.edu.cn}
\address{School of Mathematical Sciences, Fudan University,
Shanghai 200433, China.} \email{yanglingfd@fudan.edu.cn}
\thanks{Xiang Ma is supported by NSFC Project 11171004; Peng Wang is supported by NSFC Project 11201340; Ling Yang is supported by NSFC Project 11471078}

\begin{abstract}
For entire spacelike stationary 2-dimensional graphs in Minkowski spaces, we establish Bernstein type theorems under specific boundedness assumptions either on the $W$-function or on the total (Gaussian) curvature. These conclusions imply the classical Bernstein theorem for minimal surfaces in $\R^3$ and Calabi's theorem for spacelike maximal surfaces in $\R_1^3$.
\end{abstract}
\maketitle

\Section{Introduction}{Introduction}

The classical Bernstein theorem \cite{be} says that any entire minimal graph in $\R^3$ has to be an affine plane. In other words,
suppose $f:\R^2\ra \R$ is an entire solution to the minimal surface equation
\begin{equation}
\text{div}\left(\f{\n f}{\sqrt{1+|\n f|^2}}\right)=0.
\end{equation}
Then $f$ has to be affine linear.
This conclusion is generally not true in the higher codimensional case. The simplest counter-example is the minimal graph $M=\text{graph }f:=\{(x,f(x)):x\in \C\}\subset \R^4$ of an arbitrary nonlinear holomorphic function $f:\C\ra \C$.

To find a suitable generalization, usually we have to add some boundedness assumptions on the growth rate of the function $f$.
Chern-Osserman \cite{c-o} obtained such a weak version of Bernstein type theorem as follows.
Let $f=(f_1,\cdots,f_m)$ be a smooth vector-valued function from
$\R^2$ to $\R^m$. If $M=\text{graph }f:=\{(x,f(x)):x\in \R^2\}$ is a minimal graph, and
\begin{equation}
W:=\left[\det\left(\de_{ij}+\sum_{1\leq \a\leq m} \pd{f_\a}{x_i}\pd{f_\a}{x_j}\right)\right]^{1/2}
\end{equation}
is uniformly bounded, then $M$ has to be an affine plane.

This $W$-function is a significant quantity for various reasons.

Firstly, for any $f:\R^2\mapsto \R^m$ and its graph, denote the metric on $\text{graph }f$ as $g=g_{ij}dx_idx_j$ under the
global coordinate chart $x=(x_1,x_2)\mapsto (x,f(x))\in \text{graph }f$, then the area element is given by $Wdx_1\w dx_2$. Thus $W$ is a geometric measure of the area growth of the graph of $f$.

Secondly, Chern-Osserman's theorem can be stated in the language of PDE as below. Namely, the entire solution to the following PDE system
\begin{equation}\label{PDE}\aligned
\sum_{1\leq i\leq 2} \pf{x_i}(Wg^{ij})&=0\qquad j=1,2\\
\sum_{1\leq i,j\leq 2}\pf{x_i}\left(Wg^{ij}\pd{f_\a}{x_j}\right)&=0\qquad \a=1,\cdots,m
\endaligned\end{equation}
has to be affine linear provided that $W\leq C$ for a positive constant $C$, where
\begin{equation}\label{metric}
(g_{ij}):=I_2+J_f^T E J_f
\end{equation}
($I_2, E$ denote the identity matrices of size $2$ and $m$ separately, $J_f:=(\pd{f^\a}{x_i})$ is the Jacobian of $f$), $(g^{ij}):=(g_{ij})^{-1}$ and $W=\det(g_{ij})^{1/2}$.
A key point from the analytic viewpoint is that the boundedness of $W$ ensures that (\ref{PDE}) is a uniformly elliptic PDE system.

For more work on the generalization of Chern-Osserman's theorem in relation with the $W$-function, see \cite{b}, \cite{fc}, \cite{j-x-y1} and \cite{j-x-y2}.

Now we consider entire spacelike stationary graphs in Minkowski spaces.
 They also correspond to solutions to (\ref{PDE}), with the differences being that $f=(f_1,\cdots,f_m)$ is now from $\R^2$ to the $m$-dimensional Minkowski space $\R_1^m$, and that $E$ appearing in \eqref{metric} should be replaced by the Minkowski inner product matrix $\mathrm{diag}(1,1,\cdots,1,-1)$. Here we need to assume that
$(g_{ij})$ is positive-definite everywhere.

When $m=1$, $M$ becomes a spacelike maximal graph in $\R_1^3$, which has to be an affine plane. This is a well-known Berntein type result by E. Calabi \cite{ca}.
But for higher codimensional cases, the Bernstein type result fails to be true even if the $W$-function is uniformly bounded. Such a counterexample can be found in \cite{m-w-w} which is given by the function
$$f(x_1,x_2)=\left(2\sinh(x_1)\cos(-\f{\sqrt{2}}{2}x_2),
2\cosh(x_1)\cos(-\f{\sqrt{2}}{2}x_2)\right).$$
So it is a more subtle problem  about the value distribution of the $W$-function for entire spacelike stationary graphs in Minkowski spaces. This is the main topic of the present paper.

As the first step, we generalize Osserman's result in \S 5 of \cite{o} to entire spacelike stationary graphs in the Minkowski space. They are still conformally equivalent to
the complex plane (see Theorem \ref{iso}), and have an explicit simple representation formula. Based on these formulas, we establish the following results:

1) Let $M$ be an entire spacelike stationary
graph in $\R_1^4$,  then the $W$-function is either constant, or takes each values in $[r^{-1},r]$ infinitely often, where
$r$ can be any positive number strictly bigger than $1$. Moreover, $W\equiv \text{const}$ if and only if $M$ is a flat surface
(see Theorem \ref{t1}).

2) For any entire spacelike stationary
graph $M$ in $\R_1^4$, If $W\leq 1$ (or $W\geq 1$) always holds true on $M$, then $M$ has to be flat (see Corollary \ref{ber1}). Note that Calabi's theorem \cite{ca} and the classical Bernstein theorem \cite{be}
can be easily deduced from the above 2 conclusions, respectively.

3) For any entire spacelike stationary
graph $M$ in $\R_1^n (n\ge 4)$, if $W\leq 1$, then $M$ must be flat (see Theorem
\ref{ber2}). (On the contrary, the same conclusion does not necessarily hold true in the case $W\geq 1$; see Proposition \ref{ber3}.)

Another measure of the complexity of a complete stationary surface is its total Gaussian curvature $\int_M |K| dM$. It is closely related with its end behavior at the infinity (see the generalized Jorge-Meeks formula in \cite{m-w-w}).
Using the Weierstrass representation formula given in \cite{m-w-w}, one can compute the integral of the Gauss curvature and the normal curvature of an arbitrary
spacelike stationary surface in $\R_1^4$. A Bernstein type theorem (Theorem \ref{ftc}) follows immediately, which states that an entire spacelike stationary graph in $\R_1^4$
has to be flat, provided that $\int_M |K|dM<+\infty$. (This result cannot be generalized to higher codimensional cases.)

\Section{Entire graphs in Minkowski spaces and the $W$-function}{Entire graphs in Minkowski spaces and the $W$-function}\label{eg}

 Let $\R_1^m$ denote the $m$-dimensional Minkowski space. For any $\mathbf{u}=(u_1,\cdots,u_{m-1},u_m)$, $\mathbf{v}=(v_1,\cdots,v_{m-1},v_m)\in
\R_1^m$, the Minkowski inner product is given by
\begin{equation}
\lan \mathbf{u},\mathbf{v}\ran=u_1v_1+\cdots+u_{m-1}v_{m-1}-u_mv_m.
\end{equation}
Let $f: \R^2\ra \R_1^m$
\begin{equation}
(x_1,x_2)\mapsto f(x_1,x_2)=(f_1(x_1,x_2),\cdots,f_m(x_1,x_2))
\end{equation}
be a smooth vector-valued function.
As in \S 3 of \cite{o}, we introduce the
vector notation
\begin{equation}
p:=\pd{f}{x_1},\qquad q:=\pd{f}{x_2}.
\end{equation}
Let $M=\text{graph }f:=\{(x,f(x)):x\in \R^2\}$ be the entire graph in $\R_1^{2+m}$ generated by $f$, then the metric on $M$ is
\begin{equation}
g=g_{ij}dx_idx_j
\end{equation}
with
\begin{equation}
g_{11}=1+\lan p,p\ran,\quad g_{22}=1+\lan q,q\ran,\quad g_{12}=g_{21}=\lan p,q\ran.
\end{equation}

According to the properties of positive-definite matrices, $M$ is a spacelike surface if and only if $1+\lan p,p\ran>0$ and $\det(g_{ij})>0$.
Hence
\begin{equation}
W=\det(g_{ij})^{1/2}>0
\end{equation}
for any spacelike graph.

 Denote by
$\mc{P}_0$ the orthogonal projection of $\R_1^{2+m}$ onto $\R^2$, then $w:=W^{-1}$ is equivalent to the Jacobian determinant of
$\mc{P}_0|_M$. Thus $W\leq 1$ ($\equiv 1,\ \geq 1$) is equal to saying that $\mc{P}_0|_M$ is an area-increasing (area-preserving, area-decreasing) map.

For entire graphs in the Euclidean space, it is well-known that the orthogonal projection onto the coordinate plane is a length-decreasing map, which becomes
an isometry if and only if the graph is parallel to the coordinate plane. Therefore every entire graph in the Euclidean space must be complete.
But the following examples shows the above properties cannot be generalized to entire graphs in Minkowski spaces.

\textbf{Examples:}
\begin{itemize}
\item Let $\mathbf{y}_0$ be a non-zero light-like vector in $\R_1^m$, $h$ be a smooth real-valued function on $\R^2$ and $f:=h\mathbf{y}_0$,
then $p=\pd{h}{x_1}\mathbf{y}_0, q=\pd{h}{x_2}\mathbf{y}_0$ and hence $g_{ij}=\de_{ij}$, which implies the projection of $M=\text{graph }f$ onto $\R^2$
is an isometry, but $M$ cannot be an affine plane of $\R_1^{2+m}$ whenever $h$ is nonlinear.

\item Let $t\in \R\mapsto \th(t)\in (-\pi/2,\pi/2)$ be a smooth odd function, which satisfies $\lim_{t\ra +\infty}\th(t)=\pi/2$ and $\pi/2-\th(t)=O(t^{-2})$.
Denote
$$h(t):=\int_0^t \sin(\th(t))dt,$$
then $h$ is a smooth even function on $\R$. Define
$$f(x_1,x_2)=(0,\cdots,0,h(r))\qquad (r=\sqrt{x_1^2+x_2^2}),$$
then $p=\pd{f}{x_1}=(0,\cdots,0,\f{h'(r)x_1}{r})$, $q=\pd{f}{x_2}=(0,\cdots,0,\f{h'(r)x_2}{r})$ and hence
$$\aligned
g_{11}&=1+\lan p,p\ran=1-\f{h'(r)^2 x_1^2}{r^2}\geq 1-h'(r)^2=\cos^2(\th(t))>0,\\
\det(g_{ij})&=\det\left(\begin{array}{cc}
1-\f{h'(r)^2 x_1^2}{r^2} & -\f{h'(r)^2 x_1x_2}{r^2}\\
-\f{h'(r)^2x_1x_2}{r^2} & 1-\f{h'(r)^2x_2^2}{r^2}.
\end{array}\right)=1-h'(r)^2>0.
\endaligned$$
Therefore $M=\text{graph }f$ is an entire spacelike graph. Denote $\g: t\in \R\mapsto (t,0,f(t,0))$, then $\g$ is a smooth curve in $M$ tending to infinity. Since
$f(t,0)=(0,\cdots,0,h(t))$,
$$L(\g)=\int_{-\infty}^{+\infty} \sqrt{1-h'(t)^2}dt=\int_{-\infty}^{+\infty}\cos(\th(t))dt.$$
When $t\ra \infty$, $\cos(\th(t))\sim \pi/2-|\th(t)|\sim |t|^{-2}$, therefore $L(\g)<+\infty$ and hence $M$ cannot be complete.

\end{itemize}
\bigskip
\Section{Isothermal parameters of spacelike stationary graphs}{Isothermal parameters of spacelike stationary graphs}

Let $\mathbf{x}:M\ra \R_1^{2+m}$ be a spacelike surface in the Minkowski space. If the mean curvature vector field $\mathbf{H}$ vanishes everywhere, then $M$ is said to be \textit{stationary}. $M$ is stationary if and only if the restriction of any coordinate function on $M$ is harmonic. Namely,
 $\De x_l\equiv 0$ for each $1\leq l\leq 2+m$, with $\De$ the Laplace-Beltrami operator with respect to the induced metric on $M$ (see \cite{m-w-w}).

Now we additionally assume $M$ to be an entire graph over $\R^2$. More precisely, there exists $f:\R^2\ra \R_1^m$, such that $M=\text{graph }f=:\{(x,f(x)):x\in \R^2\}$.
The denotation of $p,q,g_{ij},W$ is same as in Section \ref{eg}. For an arbitrary smooth function $F$ on $M$,
\begin{equation}
\De F=W^{-1}\p_i(Wg^{ij}\p_j F),
\end{equation}
where
\begin{equation}
(g^{ij})=(g_{ij})^{-1}=W^{-2}\left(\begin{array}{cc}
1+\lan q,q\ran & -\lan p,q\ran\\
-\lan p,q\ran & 1+\lan p,p\ran
\end{array}\right).
\end{equation}
The stationarity of $M$ implies $x_1,x_2$ are both harmonic functions on $M$, hence
\begin{equation}\aligned
0&=W\De x_1=\p_i (Wg^{ij}\p_j x_1)\\
&=\p_i(Wg^{ij}\de_{1j})=\p_i (Wg^{i1})\\
&=\pf{x_1}\left(\f{1+\lan q,q\ran}{W}\right)-\pf{x_2}\left(\f{\lan p,q\ran}{W}\right)
\endaligned
\end{equation}
and similarly
\begin{equation}\aligned
0&=W\De x_2=\p_i (Wg^{i2})\\
&=-\pf{x_1}\left(\f{\lan p,q\ran}{W}\right)+\pf{x_2}\left(\f{1+\lan p,p\ran}{W}\right).
\endaligned
\end{equation}
The above 2 equations implies the existence of smooth functions $\xi_1$ and $\xi_2$, such that
\begin{equation}\aligned
\pd{\xi_1}{x_1}&=\f{1+\lan p,p\ran}{W},\qquad \pd{\xi_1}{x_2}=\f{\lan p,q\ran}{W},\\
\pd{\xi_2}{x_1}&=\f{\lan p,q\ran}{W}, \qquad \pd{\xi_2}{x_2}=\f{1+\lan q,q\ran}{W}.
\endaligned
\end{equation}

As in \S 5 of \cite{o}, one can define the Lewy's transformation $L: (x_1,x_2)\in \R^2\ra (\eta_1,\eta_2)\in \R^2$ by
\begin{equation}
\eta_i=x_i+\xi_i(x_1,x_2)\qquad i=1,2.
\end{equation}
Since the Jacobi matrix of $L$
\begin{equation}
J_L=I_2+\left(\pd{\xi_i}{x_j}\right)=I_2+W^{-1}(g_{ij})
\end{equation}
is positive-definite, $L$ is a local diffeomorphism. Again based on the fact that $\big(\pd{\xi_i}{x_j}\big)$ is positive-definite, one can proceed as in
\cite{le} or \S 5 of \cite{o} to show that $L$ is length-increasing, thus $L$ is injective. Let $\Om$ be the image of $L$, then $\Om$ is open. If $\Om\neq \R^2$,
take $\eta$ in the complement of $\Om$ that is nearest to $L(0)$, and find a sequence of points $\{\eta^{(k)}:k\in \Bbb{Z}^+\}$, such that
$|\eta^{(k)}-L(0)|<|\eta-L(0)|$ and $\lim_{k\ra \infty}\eta^{(k)}=\eta$, then there exists $x^{(k)}\in \R^2$, such that $\eta^{(k)}=L(x^{(k)})$.
Since $L$ is length-increasing, $\{x^{(k)}:k\in \Bbb{Z}^+\}$ lies in a bounded domain of $\R^2$, then there exists an subsequence converging to $x\in \R^2$,
which implies $L(x)=\eta$ and causes a contradiction. Therefore $\Om=\R^2$ and then $L$ is a diffeomorphism of $\R^2$ onto itself.

Denote by $\la_1^2, \la_2^2$ ($\la_1,\la_2>0$) the eigenvalues of $(g_{ij})$, then $W=\det(g_{ij})^{1/2}=\la_1\la_2$ and there exists an orthogonal matrix $O$, such that
$$(g_{ij})=O^T\left(\begin{array}{cc} \la_1^2 & \\ & \la_2^2\end{array}\right)O.$$
Hence
$$\aligned
J_L&=I_2+W^{-1}(g_{ij})=O^T\left(\begin{array}{cc}
1+\f{\la_1}{\la_2} & \\
& 1+\f{\la_2}{\la_1}
\end{array}\right)O\\
&=(\la_1^{-1}+\la_2^{-1})O^T\left(\begin{array}{cc}
\la_1 & \\ & \la_2\end{array}\right)O
\endaligned$$
and furthermore
$$\aligned
d\eta_1^2+d\eta_2^2&=(d\eta_1\ d\eta_2)\left(\begin{array}{c} d\eta_1 \\ d\eta_2\end{array}\right)
=(dx_1\ dx_2)J_L^T J_L\left(\begin{array}{c} dx_1 \\ dx_2\end{array}\right)\\
&=(\la_1^{-1}+\la_2^{-1})^2(dx_1\ dx_2)O^T\left(\begin{array}{cc}
\la_1^2 & \\ & \la_2^2\end{array}\right)O \left(\begin{array}{c} dx_1 \\ dx_2\end{array}\right)\\
&=(\la_1^{-1}+\la_2^{-1})^2(dx_1\ dx_2)(g_{ij}) \left(\begin{array}{c} dx_1 \\ dx_2\end{array}\right)\\
&=(\la_1^{-1}+\la_2^{-1})^2(g_{ij}dx_idx_j),
\endaligned$$
i.e.
\begin{equation}
g=g_{ij}dx_idx_j=(\la_1^{-1}+\la_2^{-1})^{-2}(d\eta_1^2+d\eta_2^2).
\end{equation}
This means that $(\eta_1,\eta_2)$ are global isothermal parameters on $M$.

Denote
\begin{equation}
\ze:=\eta_1+\sqrt{-1}\eta_2
\end{equation}
and
\begin{equation}\label{phi1}
\beta_l:=\pd{x_l}{\ze}=\f{1}{2}\Big(\pd{x_l}{\e_1}-\sqrt{-1}\pd{x_l}{\e_2}\Big)\qquad l=1,\cdots,2+m.
\end{equation}
Then the harmonicity of coordinate functions implies
$$0=\f{\p^2 x_l}{\p \ze\p \bar{\ze}}=\f{\p \beta_l}{\p \bar{\ze}},$$
i.e. $\beta_1,\cdots,\beta_{2+m}$ are all holomorphic functions on $M$. A straightforward calculation shows
$-4\text{Im}(\bar{\beta}_1\beta_2)$ equals the Jacobian of the inverse of Lewy's transformation, which is positive everywhere,
thus $\f{\beta_2}{\beta_1}=\f{\bar{\beta}_1\beta_2}{|\beta_1|^2}$ is an entire function whose imaginary part is always negative.
The classical Liouville's Theorem implies $\f{\beta_2}{\beta_1}\equiv c:=a-b\sqrt{-1}$, where $a,b\in \R$ and $b>0$. In conjunction with
(\ref{phi1}) we get
\begin{equation}\label{eta}
\aligned
\pd{x_2}{\e_1}&=a\pd{x_1}{\e_1}-b\pd{x_1}{\e_2}\\
\pd{x_2}{\e_2}&=b\pd{x_1}{\e_1}+a\pd{x_1}{\e_2}.
\endaligned
\end{equation}
Let $(u_1,u_2)$ be global parameters of $M$, satisfying $x_1=u_1$ and $x_2=au_1+bu_2$. Then (\ref{eta}) tells us
\begin{equation}
\pd{u_2}{\e_1}=-\pd{u_1}{\e_2},\quad \pd{u_2}{\e_2}=\pd{u_1}{\e_1}.
\end{equation}
This means the one-to-one map $(\eta_1,\eta_2)\in \R^2\mapsto (u_1,u_2)\in \R^2$ is bi-holomorphic. Thereby we arrive at the following conclusion:

\begin{thm}\label{iso}
Let $f:\R^2\ra \R_1^m$ be a smooth vector-valued function, such that $M=\text{graph }f:=\{(x,f(x)):x\in \R^2\}$ is a spacelike stationary surface,
then there exists a nonsingular linear transformation
\begin{equation}\label{tran}\aligned
x_1&=u_1\\
x_2&=au_1+bu_2,\quad (b>0)
\endaligned
\end{equation}
such that $(u_1,u_2)$ are global isothermal parameters for $M$.
\end{thm}

Now we introduce the complex coordinate $z:=u_1+\sqrt{-1}u_2$ and denote
\begin{equation}
\alpha=(\alpha_1,\cdots,\alpha_{2+m}):=\pd{\mathbf{x}}{z}=\f{1}{2}\Big(\pd{\mathbf{x}}{u_1}-\sqrt{-1}\pd{\mathbf{x}}{u_2}\Big).
\end{equation}
then $\alpha$ is a holomorphic vector-valued function. The induced metric on $M$ can be written as
\begin{equation*}
\aligned
g&=\lan \pd{\mathbf{x}}{z},\pd{\mathbf{x}}{z}\ran dz^2+\lan \pd{\mathbf{x}}{\bar{z}},\pd{\mathbf{x}}{\bar{z}}\ran d\bar{z}^2+2\lan \pd{\mathbf{x}}{z},\pd{\mathbf{x}}{\bar{z}}\ran|dz|^2\\
&=2\text{Re}\big(\lan \alpha,\alpha\ran dz^2\big)+2\lan \alpha,\bar{\alpha}\ran|dz|^2.
\endaligned
\end{equation*}
Here $|dz|^2:=\f{1}{2}(dz\otimes d\bar{z}+d\bar{z}\otimes dz)=du_1^2+du_2^2$. Since $(u_1,u_2)$ are isothermal parameters for $M$,
\begin{equation}\label{phi2}
\lan \alpha,\alpha\ran=0
\end{equation}
and hence
\begin{equation}
g=2\lan \alpha,\bar{\alpha}\ran |dz|^2.
\end{equation}
Noting that $\alpha_1=\pd{x_1}{z}=\f{1}{2}$, $\alpha_2=\pd{x_2}{z}=\f{1}{2}(a-b\sqrt{-1})=\f{c}{2}$, (\ref{phi2}) equals to say
\begin{equation}\label{phi3}
\alpha_{2+m}^2=\alpha_1^2+\cdots+\alpha_{1+m}^2=\f{1+c^2}{4}+\alpha_3^2+\cdots+\alpha_{1+m}^2.
\end{equation}
Thus
$$\aligned
\lan \alpha,\bar{\alpha}\ran&=|\alpha_1|^2+\cdots+|\alpha_{1+m}|^2-|\alpha_{2+m}|^2\\
&=\f{1+|c|^2}{4}+|\alpha_3|^2+\cdots+|\alpha_{1+m}|^2-\big|\f{1+c^2}{4}+\alpha_3^2+\cdots+\alpha_{1+m}^2\big|\\
&\geq \f{1+|c|^2-|1+c^2|}{4}
\endaligned$$
and moreover
\begin{equation}
g\geq \f{1+|c|^2-|1+c^2|}{2}|dz|^2
\end{equation}
Observing that $1+|c|^2-|1+c^2|>0$ is a direct corollary of $b>0$, we get a conclusion as follows.

\begin{cor}
Let $M=\text{graph }f:=\{(x,f(x)):x\in \R^2\}$ be a spacelike stationary graph generated by $f:\R^2\ra \R_1^m$,
 then the induced metric on $M$ is complete.
\end{cor}

(\ref{tran}) implies $dx_1\w dx_2=bdu_1\w du_2$, hence
$$\aligned
dM&=2\lan \alpha,\bar{\alpha}\ran du_1\w du_2\\
&=2b^{-1}\lan \alpha,\bar{\alpha}\ran dx_1\w dx_2\\
&=\f{1+|c|^2+4(|\alpha_3|^2+\cdots+|\alpha_{1+m}|^2-|\alpha_{2+m}|^2)}{2b}dx_1\w dx_2
\endaligned$$
with $dM$ the area element of $M$. In other words,
\begin{equation}\label{W}
W=\f{1+|c|^2+4(|\alpha_3|^2+\cdots+|\alpha_{1+m}|^2-|\alpha_{2+m}|^2)}{2b}.
\end{equation}
\bigskip

\Section{On $W$-functions for entire stationary graphs in $\R_1^4$}{On $W$-functions for entire stationary graphs in $\R_1^4$}

\begin{thm}\label{t1}
Let $f:\R^2\ra \R_1^2$ be a smooth function, such that $M=\text{graph }f:=\{(x,f(x)):x\in \R^2\}$ is a spacelike stationary graph,
then one and only one of the following three cases must occur:

(i) $f$ is affine linear and $W\equiv r$, where $r$ is an arbitrary positive constant;

(ii) $f=h\mathbf{y}_0+\mathbf{y}_1$ with $h$ a nonlinear harmonic function on $\R^2$, $\mathbf{y}_0$ a nonzero lightlike vector in $\R_1^2$
and $\mathbf{y}_1$ a constant vector, and
$W\equiv 1$;

(iii) $W$ takes each values in $[r^{-1},r]$ infinitely often, where $r$ is an arbitrary number in $(1,+\infty)$.
\end{thm}

\begin{proof}
(\ref{phi2}) is equal to
\begin{equation}\label{phi4}
\alpha_3^2-\alpha_4^2=-(\alpha_1^2+\alpha_2^2)=-\f{1+c^2}{4}
\end{equation}
and (\ref{W}) gives
\begin{equation}\label{W2}
W=\f{1+|c|^2+4(|\alpha_3|^2-|\alpha_4|^2)}{2b}.
\end{equation}

If $\alpha_3$ is a constant function, then (\ref{phi4}) shows $\alpha_4$ is also constant, and
$$x_a(z)=\text{Re}\int_0^z \alpha_a dz+x_a(0)\qquad \forall a=3,4$$
is affine linear. Hence $f$ is affine linear and $W\equiv r$,
where $r$ can be taken to be any value in $(0,\infty)$.  This is Case (i).

Now we assume $\alpha_3$ is non-constant, then (\ref{phi4}) implies $\alpha_4$ is also non-constant.

If $c=-\sqrt{-1}$, then (\ref{phi4}) gives
$$0=\alpha_3^2-\alpha_4^2=(\alpha_3+\alpha_4)(\alpha_3-\alpha_4).$$
Noting that the zeros of a non-constant holomorphic function have to be isolated, we get
$\alpha_3+\alpha_4=0$ or $\alpha_3-\alpha_4=0$. Thus $|\alpha_3|=|\alpha_4|$
and then (\ref{W2}) shows $W\equiv 1$. Let $\beta$ be the unique holomorphic function such that
$\beta'=\alpha_3$ and $\beta(0)=0$, then $\alpha_3\pm \alpha_4=0$ implies
$$\aligned
f(x_1,x_2)&=(x_3(u_1,u_2),x_4(u_1,u_2))=(x_3(z),x_4(z))\\
&=\text{Re}\int_0^z (\alpha_3,\alpha_4)dz+(x_3(0),x_4(0))\\
&=\text{Re }\beta(z)(1,\mp 1)+f(0,0).
\endaligned$$
Now we put $h:=\text{Re }\beta(z)$, $\mathbf{y}_0:=(1,\mp 1)$ and $\mathbf{y}_1:=f(0,0)$, then $h$ is a nonlinear harmonic function,
$\mathbf{y}_0$ is a light-like vector and $f=h\mathbf{y}_0+\mathbf{y}_1$. This is Case (ii).

Otherwise $c\neq -\sqrt{-1}$ and hence $-\f{1+c^2}{4}\neq 0$. Let
$\mu\neq 0$ such that $\mu^2=-\f{1+c^2}{4}$, and $h_1,h_2$ be
holomorphic functions such that $\alpha_3=\mu h_1$, $\alpha_4=\mu
h_2$, then $\mu^2(h_1^2-h_2^2)=\alpha_3^2-\alpha_4^2=\mu^2$ gives
$$1=h_1^2-h_2^2=(h_1+h_2)(h_1-h_2),$$
which implies $h_1+h_2$ is an entire function containing no zero. Hence there exists an entire function $\beta$, such that
$h_1+h_2=e^\beta$, then $h_1-h_2=e^{-\beta}$ and hence
\begin{equation}
h_1=\cosh \beta,\qquad h_2=\sinh \beta.
\end{equation}
By computing,
$$\aligned
&|h_1|^2-|h_2|^2=|\cosh \beta|^2-|\sinh \beta|^2\\
=&\f{1}{2}(e^{\beta-\bar{\beta}}+e^{-\beta+\bar{\beta}})=\f{1}{2}(e^{2\text{Im}\beta\sqrt{-1}}+e^{-2\text{Im}\beta\sqrt{-1}})\\
=&\cos(2\text{Im}\beta)
\endaligned$$
and hence
\begin{equation}
\aligned
W&=\f{1+|c|^2+4(|\alpha_3|^2-|\alpha_4|^2)}{2b}\\
&=\f{1+|c|^2+4|\mu|^2(|h_1|^2-|h_2|^2)}{2b}\\
&=\f{1+|c|^2+|1+c^2|\cos(2\text{Im}\beta)}{2b}.
\endaligned
\end{equation}
Denote $r_1:=\inf W=\f{1+|c|^2-|1+c^2|}{2b}$ and $r_2:=\sup W=\f{1+|c|^2+|1+c^2|}{2b}$.
Due to the Picard theorem, $W$ takes each values in $[r_1,r_2]$ infinitely often. Noting that $c=a-b\sqrt{-1}$,
one computes
$$\aligned
r_1r_2&=\f{(1+|c|^2)^2-|1+c^2|^2}{4b^2}=\f{1+2|c|^2+|c|^4-(1+c^2+\bar{c}^2+|c|^4)}{4b^2}\\
&=\f{4b^2}{4b^2}=1.
\endaligned$$
Hence $r_1\in (0,1)$ and $r_2\in (1,+\infty)$.

Now we take $b:=1$, then $c=a-\sqrt{-1}$ and
$r_2=\f{1}{2}(2+a^2+|a|\sqrt{a^2+4})$. Denote $\mu:t\in \R^+\mapsto \mu(t)=\f{1}{2}(2+t^2+|t|\sqrt{t^2+4})$,
then $\mu$ is a strictly increasing function and $\lim_{t\ra 0}\mu(t)=1$, $\lim_{t\ra +\infty}\mu(t)=+\infty$.
Hence for an arbitrary number $r\in (1,\infty)$, one can find $a\in \R^+$, such that $r_2=r$ and then $W$ takes each values
in $[r^{-1},r]$ infinitely often. This is Case (iii).

\end{proof}

\begin{cor}\label{ber1}
Let $M$ be an entire spacelike stationary graph in $\R_1^{4}$ generated by a smooth function $f=(f_1,f_2):\R^2\ra \R_1^2$, if $W\leq 1$ (resp. $W\geq 1$),
then $f$ is affine linear or $f=h\mathbf{y}_0+\mathbf{y}_1$, with $h$ a nonlinear harmonic function, $\mathbf{y}_0$ a nonzero lightlike
vector and $\mathbf{y}_1$ a constant vector. Moreover, $W>1$ (resp. $W<1$) forces $f$ to be affine linear, representing an affine plane
in $\R_1^4$.

\end{cor}

\noindent \textbf{Remark. }If $f_2\equiv 0$, then $M=\text{graph }f$ is a minimal entire graph in $\R^3$ and $W\geq 1$. By Corollary \ref{ber1},
$f$ is affine linear or $f=h\mathbf{y}_0+\mathbf{y}_1$, where $h$ is a nonlinear harmonic function and $\mathbf{y}_0$ is a nonzero
lightlike vector. But $f_2\equiv 0$ denies the occurence of the latter case. Hence $f$ is an affine linear function and thereby the
classical Bernstein theorem \cite{be} can be derived from Corollary \ref{ber1}. Similarly, Corollary \ref{ber1} implies any
 spacelike maximal entire graph in $\R_1^3$ has to be affine linear. This is Calabi's theorem \cite{ca}.

\Section{Bernstein type theorems for entire stationary graphs in $\R_1^{2+m}$}{Bernstein type theorems for entire stationary graphs in $\R_1^{2+m}$}

It is natural to ask whether one can generalized the conclusion of Corollary \ref{ber1} to higher codimensional cases.

For the first statement, i.e. $W\leq 1$, the answer is ``yes":

\begin{thm}\label{ber2}
Let $f:\R^2\ra \R_1^m$ be a smooth function, such that $M=\text{graph }f:=\{(x,f(x)):x\in \R^2\}$ is a spacelike stationary graph in $\R_1^{2+m}$.
If the orthogonal projection $\mc{P}_0$ of $M$ onto the coordinate plane $\R^2$ is area-increasing (i.e. $W\leq 1$), then $f$ is affine linear or $f=h\mathbf{y}_0+\mathbf{y}_1$, with $h$ a nonlinear harmonic function, $\mathbf{y}_0$ a nonzero lightlike
vector and $\mathbf{y}_1$ a constant vector. Moreover, if $\mc{P}_0$ is strictly area-increasing (i.e. $W<1$), then $f$ has to be affine linear and $M$
is an affine plane.
\end{thm}

\begin{proof}
We shall consider the problem in the following four cases.

\textbf{Case I.} $\alpha_1,\cdots,\alpha_{2+m}$ are all constant functions.

As in the proof of Theorem \ref{t1}, one can show $f$ is an affine linear function.

\textbf{Case II.} $\alpha_{2+m}$ is a constant function, but $\alpha_l$ is non-constant for some $1\leq l\leq 1+m$.

By the classical Liouville Theorem, there exists a point in $\C$, such that $|\alpha_{l}|^2\geq |\alpha_{2+m}|^2+b-\f{1}{4}(1+|c|^2)$
at this point. Combing with (\ref{W}) gives
$$\aligned
W&=\f{1+|c|^2+4(|\alpha_3|^2+\cdots+|\alpha_{1+m}|^2-|\alpha_{2+m}|^2)}{2b}\\
&\geq \f{1+|c|^2+4(|\alpha_l|^2-|\alpha_{2+m}|^2)}{2b}\geq 2.
\endaligned$$
This gives a contradiction to the assumption that $W\leq 1$ everywhere. Hence the case cannot occur.

\textbf{Case III.} $\alpha_{2+m}$ is non-constant and $c\neq
-\sqrt{-1}$.

$c\neq \sqrt{-1}$ implies
$$\f{1+|c|^2}{2b}=\f{1+b^2+a^2}{2b}>1.$$
Denote $\de:=\f{1+|c|^2}{2b}-1$. Again the classical Liouville Theorem implies the existence of a point, such that
$|\alpha_{2+m}|^2<\f{1}{2}b\de$ at this point. Hence
$$\aligned
W&=\f{1+|c|^2+4(|\alpha_3|^2+\cdots+|\alpha_{1+m}|^2-|\alpha_{2+m}|^2)}{2b}\\
&\geq \f{1+|c|^2-4|\alpha_{2+m}|^2}{2b}>1+\de-\f{4\cdot \f{1}{2}b\de}{2b}=1,
\endaligned$$
which causes a contradiction and therefore this case cannot happen.

\textbf{Case IV.} $\alpha_{2+m}$ is non-constant and $c=-\sqrt{-1}$.

Let $h_3,\cdots,h_{1+m}$ be meromorphic functions, such that
$$\alpha_3^2=h_3\alpha_{2+m}^2,\cdots,\alpha_{1+m}^2=h_{1+m}\alpha_{2+m}^2.$$
Then (\ref{phi3}) tells us
$$\aligned
\alpha_{2+m}^2&=\f{1+c^2}{4}+\alpha_3^2+\cdots+\alpha_{1+m}^2=\alpha_3^2+\cdots+\alpha_{1+m}^2\\
&=(h_3+\cdots+h_{1+m})\alpha_{2+m}^2.
\endaligned$$
Since $\alpha_{2+m}$ is a non-constant function, we have
$$h_3+\cdots+h_{1+m}\equiv 1.$$
Due to the triangle inequality,
$$\aligned
W&=\f{1+|c|^2+4(|\alpha_3|^2+\cdots+|\alpha_{1+m}|^2-|\alpha_{2+m}|^2)}{2b}\\
&=1+2(|\alpha_3^2|+\cdots+|\alpha_{1+m}^2|-|\alpha_{2+m}^2|)\\
&=1+2(|h_3|+\cdots+|h_{1+m}|-1)|\alpha_{2+m}|^2\geq 1
\endaligned$$
and the equality holds if and only if the functions $h_3,\cdots,h_{1+m}$ all take values in $\R^+\cup \{0,\infty\}$.
Again using the Liouville Theorem, we know that $h_3,\cdots,h_{1+m}$ are all constant real functions. Therefore, there
exist $v_3,\cdots,v_{1+m}\in \R$, such that $v_3^2+\cdots+v_{1+m}^2=1$ and
$$(\alpha_3,\cdots,\alpha_{1+m},\alpha_{2+m})=(v_3,\cdots,v_{1+m},1)\alpha_{2+m}.$$
Let $\beta$ be the unique holomorphic function such that $\beta'=\alpha_{2+m}$ and $\beta(0)=0$. Denote $h:=\text{Re }\beta$,
$\mathbf{y}_0:=(v_3,\cdots,v_{1+m},1)$ and $\mathbf{y}_1:=f(0,0)$, then $h$ is a non-linear harmonic function and
$\mathbf{y}_0$ is a light-like vector. We can proceed as in the proof of Theorem \ref{t1} to show $f=h\mathbf{y}_0+\mathbf{y}_1$.
Note that in this case $W\equiv 1$.

\end{proof}

But our answer is ``no" for the second statement, i.e. $W\geq 1$. In fact, we have the following result:

\begin{pro}\label{ber3}
For any real number $C\geq 1$ and $\ep>0$, there exists an entire spacelike stationary graph in $\R_1^{2+m}$ ($m\geq 3$) generating by
$f:\R^2\ra \R_1^m$, such that $\inf W\cdot \sup W=C$ and $0<\sup W-\inf W<\ep$.
\end{pro}

\begin{proof}
Now we put $c:=-b\sqrt{-1}$ and let $d$ be a real number to be chosen. Let $\mu$ be a complex number such that
$$\mu^2=-\f{1+c^2+d^2}{4}=-\f{1-b^2+d^2}{4}.$$
Denote
$$\aligned
&\alpha_1=\f{1}{2},\alpha_2=\f{c}{2}=-\f{b}{2}\sqrt{-1},\alpha_3=\cdots=\alpha_{m-1}=0,\\
&\alpha_m=\f{d}{2},\alpha_{1+m}=\mu \cosh z,\alpha_{2+m}=\mu \sinh z.
\endaligned$$
Since
$$\lan \alpha,\alpha\ran=\alpha_1^2+\alpha_2^2+\alpha_m^2+\alpha_{1+m}^2-\alpha_{2+m}^2=0$$
and $\lan \alpha,\bar{\alpha}\ran>0$, $z\mapsto \mathbf{x}(z)=\int_0^z \alpha(z)$ gives an entire spacelike stationary graph in $\R_1^{2+m}$.

As in the proof of Theorem \ref{t1}, a similar calculation shows
$$\aligned
W&=\f{1+|c|^2+4(|\alpha_3|^2+\cdots+|\alpha_{1+m}|^2-|\alpha_{2+m}|^2)}{2b}\\
&=\f{1+b^2+d^2+|1-b^2+d^2|\cos(2\text{Im}z)}{2b}
\endaligned$$
Denote $r_1:=\inf W$, $r_2:=\sup W$, then $r_1=\f{1+b^2+d^2-|1-b^2+d^2|}{2b}$,
$r_2=\f{1+b^2+d^2+|1-b^2+d^2|}{2b}$ and
$$\aligned
&r_1r_2=\f{(1+b^2+d^2)^2-(1-b^2+d^2)^2}{4b^2}=1+d^2,\\
&r_2-r_1=\f{|1-b^2+d^2|}{b}.
\endaligned$$
Now we put $d:=\sqrt{C-1}$, then $r_1r_2=C$, and one can choose $b$ being sufficiently close to $\sqrt{C}$, such that
$r_2-r_1=\f{|1-b^2+d^2|}{b}=\f{|C-b^2|}{b}\in (0,\ep)$.

\end{proof}

\Section{Stationary graphs with finite total curvature}{Stationary graphs with finite total curvature}

As demonstrated in \cite{m-w-w}, the Bernstein theorem can not be generalized directly to stationary graphs in $\R^4_1$, because one can easily construct complete stationary graphs in $\R^4_1$ which is not flat. Interestingly these examples have infinite total curvature.

On the other hand, examples of complete stationary surfaces with finite total curvature are abundant, and there holds a generalized Jorge-Meeks formula about their total Gaussian curvature (and the total normal curvature) provided that they are algebraic \cite{m-w-w}. Thus one is naturally interested to know whether there could be a stationary graph with finite total curvature. The answer to this question is the following Bernstein type theorem (Note that here we do not need the algebraic assumption).

\begin{thm}\label{ftc}
Let $f=(f_1,f_2):\R^2\ra \R_1^2$ be a smooth function, such that $M=\text{graph }f:=\{(x,f(x)):x\in \R^2\}$ is a spacelike stationary graph in $\R_1^4$ whose curvature integral $\int_M |K| dM$ converges absolutely. Then $f$ is affine linear or $f=h\mathbf{y}_0+\mathbf{y}_1$, with $h$ a nonlinear harmonic function, $\mathbf{y}_0$
a nonzero lightlike vector and $\mathbf{y}_1$ a constant vector. In both cases,
 $M$ is flat, i.e. $K\equiv 0$.
\end{thm}
\begin{proof}
Denote $z=u_1+\sqrt{-1} u_2$ as before. As in the proof of Theorem \ref{t1}, if $M$ is not a flat surface as we claimed, then the holomorphic differential
$\pd{\mathbf{x}}{z}$ can be expressed as
\begin{equation}\label{g}
(\alpha_1,\alpha_2,\alpha_3,\alpha_4)=
(\frac{1}{2},\frac{c}{2},\mu\cosh\be,\mu\sinh\be)
\end{equation}
where $c=a-b\sqrt{-1}$ is a complex constant with $b>0$, $\mu^2=-\frac{1+c^2}{4}$, and $\be=\be(z)$ is a non-constant homolomorphic function defined on $\C$. We will derive contradiction from this assumption.

By the Weierstrass representation formula given in \cite{m-w-w}, $\pd{\mathbf{x}}{z}$ can be expressed in terms of a pair of meromorphic functions $\phi,\psi$ (the \emph{Gauss maps}) and a holomorphic differential $dh=h'(z)dz$ (the \emph{height differential}) as below:
\begin{equation}
(\alpha_1,\alpha_2,\alpha_3,\alpha_4)
=(\phi+\psi,-\sqrt{-1}(\phi-\psi),1-\phi\psi,1+\phi\psi)h'.
\end{equation}
Comparing these two formulas, we obtain
\[
h'=\frac{\mu}{2}e^{\be},~~
\phi=\frac{1+c\sqrt{-1}}{2\mu}e^{-\be},~~
\psi=\frac{1-c\sqrt{-1}}{2\mu}e^{-\be}.
\]
Note that $\frac{1+c\sqrt{-1}}{2\mu}\cdot \frac{1-c\sqrt{-1}}{2\mu}=-1$,
and $b>0$ implies
$|\frac{1+c\sqrt{-1}}{2\mu}|>|\f{1+\bar{c}\sqrt{-1}}{2\mu}|$.
Denote $\frac{1+c\sqrt{-1}}{2\mu}:=re^{\sqrt{-1}\theta}$ with $r>1$ and $\th\in \R$, then $\frac{1-c\sqrt{-1}}{2\mu}=-r^{-1}e^{-\sqrt{-1}\theta}.$

In \cite{m-w-w} the Gaussian curvature and the normal curvature of a stationary surface was unified in a single formula in terms of $\phi,\psi$ and the Laplacian with respect to the
induced metric $g:=e^{2\om}|dz|^2$ as follows:
\begin{equation}\label{k1}
-K+\sqrt{-1}K^{\perp}
=\Delta\ln(\phi-\bar\psi)=4e^{-2\om}\frac{\phi_z\bar{\psi}_{\bar{z}}}{(\phi-\bar{\psi})^2}.
\end{equation}
Denote $\be:=v_1+\sqrt{-1}v_2$, where $v_1,v_2$ are both real functions on $\C$, then
\begin{equation}\label{k2}
\aligned
|K|e^{2\om}&=4\left|\text{Re}\left(\frac{\phi_z\bar{\psi}_{\bar{z}}}{(\phi-\bar{\psi})^2}\right)\right|\\
&=4\left|\text{Re}\left(\f{e^{2\sqrt{-1}\th}e^{-\be-\bar{\be}}}{(re^{\sqrt{-1}\th}e^{-\be}+r^{-1}e^{\sqrt{-1}\th}e^{-\bar{\be}})^2}\right)\right||\be'(z)|^2\\
&=4\left|\text{Re}\left(\f{1}{(re^{\f{1}{2}(\bar{\be}-\be)}+r^{-1}e^{\f{1}{2}(\be-\bar{\be})})^2}\right)\right||\be'(z)|^2\\
&=\f{4[2+(r^2+r^{-2})\cos(2v_2)]}{|re^{-\sqrt{-1}v_2}+r^{-1}e^{\sqrt{-1}v_2}|^2}|\be'(z)|^2\\
&\geq \f{4[2+(r^2+r^{-2})\cos(2v_2)]}{|r-r^{-1}|^2}|\be'(z)|^2.
\endaligned
\end{equation}
Thus the assumption of finite total curvature is equivalent to saying that
\begin{equation}\aligned\label{t-c}
+\infty>\int_M |K|dM
&=\int_{\C}
|K|e^{2\om}
du_1\wedge du_2\\
&\geq\int_{\C} \f{4[2+(r^2+r^{-2})\cos(2v_2)]}{|r-r^{-1}|^2}|\be'(z)|^2
du_1\wedge du_2\\
&\geq\int_{\C} \f{4[2+(r^2+r^{-2})\cos(2v_2)]}{|r-r^{-1}|^2}dv_1\w dv_2,
\endaligned\end{equation}
where the final equality follows from the assumption that $\beta$ is a non-constant entire function over $\mathbb{C}$, which takes almost every value of $\mathbb{C}$ at least one time.
It is easily-seen that the right hand side of (\ref{t-c}) is divergent, contradicting with the finiteness of the total curvature.

\end{proof}

\noindent{\bf Remarks:}
\begin{itemize}
\item Taking the imaginary part of (\ref{k1}), one can proceed as in (\ref{k2})-(\ref{t-c}) to get a contradiction when the condition ``$\int_M |K| dM<+\infty$" is replaced by
``$\int_M |K^\bot| dM<+\infty$". Therefore, if $M\subset \R_1^4$ is an entire spacelike stationary graph over $\R^2$, whose normal curvature integral converges absolutely,
then $M$ has to be a flat surface.
\item Let $M$ be a non-compact surface equipped with complete metric. If $\int_M |K| dM<+\infty$, then there is a compact Riemann surface $\bar{M}$, such that $M$ is
    conformally equivalently to $\bar{M}\backslash \{p_1,p_2,\cdots,p_r\}$, with $p_1,\cdots,p_r\in \bar{M}$. This is a purely intrinsic result, discovered by A. Huber \cite{hu}.
    Moreover, if we additionally assume $M$ to be a minimal surface in $\R^{2+m}$ ($m$ is arbitrary), then the Gauss map of $M$ is algebraic, and verse visa (see Theorem 1 of \cite{c-o}).
    But this conclusion is no longer true for spacelike stationary surfaces in $\R_1^4$, due to the examples with finite totally curvature and essential singularities (see \cite{m-w-w}). Hence different from the $\R^4$ case \cite{o}, the conclusion of Theorem \ref{ftc} cannot be deduced directly from (\ref{g}).

\item Combing with Theorem 1 of \cite{c-o} and \S 5 of \cite{o}, it is easy to conclude that $M=\text{graph }f:=\{(x,f(x)):x\in \R^2\}$ is a minimal surface in $\R^4$ with finite total curvature if and only if
$f=p(z)$ or $p(\bar{z})$, with $p$ an arbitrary polynomial. Noting that any minimal graph in $\R^4$ over $\R^2$ can be regarded as a spacelike stationary graph in
$\R_1^n$ ($n\geq 5$), the conclusion of Theorem \ref{ftc} cannot be generalized to spacelike stationary graphs in higher dimensional Minkowski spaces.
\end{itemize}

\bibliographystyle{amsplain}

\end{document}